\documentclass[12pt]{article} 
\usepackage[utf8x]{inputenc}
\usepackage{tikz} %%% parallel arrows 
\usepackage[all]{xy}
\usepackage{authblk}
\usepackage[babel=true]{csquotes}%%% guillemets
\usepackage{amsmath,amsthm, amssymb,amsfonts}
\usepackage[hmargin=1in,vmargin=1in]{geometry}
\numberwithin{equation}{subsection} %%% ajoute une numérotation des équations avec les sections %%%
\xyoption{arc}
\usepackage{eucal}%%% mathcal change
\usepackage{indentfirst}%%% premier ligne indenter 
\usepackage{mathrsfs}
\usepackage{verbatim}
\usepackage{makeidx}%%%% index 
\usepackage{graphicx}
\usepackage{hyperref}
\usepackage{enumitem} % % % % 
\usepackage{extpfeil} %%%% sert à nommer les fleches avec deux tetes dans un sens : fibrations...
\usepackage{dsfont}
\usepackage[T1]{fontenc}
\newtheorem{thm}{Theorem}[section]
\newtheorem{prop}[thm]{Proposition}
\newtheorem{lem}[thm]{Lemma}
\newtheorem*{thmsans}{Theorem}%%% thm sans 
\newtheorem{cor}[thm]{Corollary}

\newtheoremstyle{bidule}% name
{3pt}% Space above
{3pt}% Space below
{}% Body font
{}% Indent amount
{\scshape}% Theorem head font
{.}% Punctuation after theorem head
{.5em}% Space after theorem head
{}% Theorem head spec (can be left empty, meaning `normal')
\newtheorem{df}[thm]{Definition}
\theoremstyle{definition}
\newtheorem{rmk}[thm]{Remark}
\newtheorem{ex}[thm]{Example}

\newtheorem*{quest}{Question} % % % Question
\newtheorem*{note}{Note}

\newtheorem*{warn}{Warning}
\newtheorem*{claim}{Claim}
\newtheorem{nota}[thm]{Notation}

\newcommand{\C}{\mathcal{C}}
\newcommand{\Ub}{\mathcal{U}}%%% Ub comme Oublie%%% forgetful functor%%%%%
\newcommand{\F}{\mathcal{F}}

\newcommand{\Ar}{\text{Arr}}

\newcommand{\Ba}{\mathcal{B}}

\newcommand{\Aa}{\mathcal{A}}

\newcommand{\M}{\mathscr{M}}
\newcommand{\V}{\mathbb{V}}
\newcommand{\Ja}{\mathbf{J}} %%acyclic cofibrations%%%
\newcommand{\J}{\mathcal{J}} %% index category for limit and colimits%%%

\newcommand{\W}{\mathscr{W}}
\newcommand{\Fa}{\mathcal{F}}
\newcommand{\I}{\mathbf{I}}
\newcommand{\G}{\mathcal{G}}

\renewcommand{\S}{\mathcal{S}}

\newcommand{\Fb}{\mathbf{F}}%%%%%% Free functor 
\renewcommand{\to}{\longrightarrow}
\newcommand{\ol}{\overline}
\newcommand{\ul}{\underline}

\newcommand{\U}{\mathbb{U}}

\newcommand{\Ob}{\text{Ob}}% set of objects  
\newcommand{\tx}{\text}

\renewcommand{\to}{\longrightarrow}
\DeclareMathOperator\Id{Id}
\DeclareMathOperator\Hom{Hom}

\DeclareMathOperator\Set{\textbf{Set}} % Category of Set
\DeclareMathOperator\Lax{Lax}%% Lax functors 
\DeclareMathOperator\Cat{\mathbf{Cat}}%% Cat
\DeclareMathOperator\colim{colim}%% colim
\DeclareMathOperator\msx{\M_{\S}(\tx{$X$})}%% Pre-Coseg-Cat
\DeclareMathOperator\msxex{\M_{\S}(\tx{$X$})_{\tx{ex}}}%% Pre-Coseg-Cat

\DeclareMathOperator\mcatx{\M\tx{-}\Cat(\tx{$X$})}%% 
\DeclareMathOperator\msxproj{\M_{\S}(\tx{$X$})_{\tx{proj}}} % % % msx with the projective model structure.
\DeclareMathOperator\sx{\S_{\ol{X}}} % % % %  sx
\DeclareMathOperator\Ho{\mathbf{ho}} % % % %  homotopy category
\DeclareMathOperator\sxop{(\S_{\ol{X}})^{2\tx{-op}}} % % % %  all  morphisms  % \cite{DHKS}
\DeclareMathOperator\Depiop{\Delta_{\tx{epi}}^{op}}
\DeclareMathOperator\kc{\Pi_{\C}(\M)} 
\DeclareMathOperator\Wa{\textbf{W}_{ex}}

\DeclareMathOperator\lr{\mathbf{lr}} % 
\DeclareMathOperator\cb{\mathbf{c}}  
 
\DeclareMathOperator\laxlatch{\mathbf{Latch}_{lax}}  %% laxlatch 
\DeclareMathOperator\laxmatch{\mathbf{Match}_{lax}}  %% laxlatch 
\DeclareMathOperator\match{\mathbf{Match}}  %%
\DeclareMathOperator\latch{\mathbf{Latch}}  %% laxlatch  
\DeclareMathOperator\Depi{\Delta_{\tx{epi}}} %%
\DeclareMathOperator\sset{\mathbf{sSet}} %% simplicial sets

\DeclareMathOperator\disc{\mathbf{disc}} %%
\DeclareMathOperator\indisc{\mathbf{indisc}} %%
\DeclareMathOperator\scolax{\mathbf{SColax}}
\DeclareMathOperator\phepi{\Phi_{\tx{epi}}}
\DeclareMathOperator\phepiop{\Phi^{op}_{\tx{epi}}}
\title{Toward a stritification theorem for co-Segal categories}
\author{Hugo V. Bacard 
}
 \affil{Western University}
\date{}
\begin{document}
\maketitle

\begin{abstract}
We show that for a monoidal model category $\M=(\ul{M}, \otimes, I)$, certain co-Segal  $\M$-categories are equivalent to strict ones. 
\end{abstract}
\setcounter{tocdepth}{1}
\tableofcontents 

\section{Introduction}
We've started a theory of homotopy enrichment with the notion of co-Segal category (see \cite{COSEC1}). The basic idea is to replace the composition operation `$\C(A,B) \otimes \C(B,C) \to \C(A,C)$' by  configurations of the form:
\[
\xy
(-15,0)*+{\C(A,B) \otimes \C(B,C)}="X";
(30,0)*+{\C(A,B,C)}="Y";
(30,18)*+{\C(A,C)}="E";
{\ar@{->}^-{\varphi}"X";"Y"};
{\ar@{->}^-{}_{\wr}"E";"Y"};
{\ar@{.>}^-{}"X";"E"};
\endxy
\]
where the vertical map is a weak equivalence.\\

Such structure is defined as a lax functor $\C: \sxop \to \M$ satisfying a homotopy condition (vertical maps being weak equivalence). Here $\sxop$ is a strict $2$-category build out of a set $X$ and $\M$ is a symmetric monoidal model category (viewed as a $2$-category with a single object). The set $X$ is the set of objects of  $\C$ .\\

The philosophy of co-Segal categories is to reserve the Segal situation, but this is not the only difference. Indeed, Segal categories are defined by simplicial diagrams satisfying some homotopy conditions (Segal conditions) whereas the definition of co-Segal categories mixes both simplicial structure, homotopy conditions together with  algebraic data (the map $\varphi$ above and his cousins). And it seems that the the presence of algebraic data creates some \emph{obstruction} to have a nice homotopical understanding of these structures. 

For example, as far as the author knows, we cannot guarantee that the category of $\M$-valued lax functors inherits the left properness of $\M$. \\
Strict $\M$-categories correspond to co-Segal categories that are \emph{purely algebraic}, in the sense that the simplicial structure and homotopy conditions are trivial: everything is given by identity morphisms. \\

In this paper we investigate the \emph{strictification problem} for co-Segal $\M$-categories for a (symmetric) monoidal model category $\M$. So morally we try to find an analogue of Bergner's strictification theorem for Segal categories (see \cite{Bergner_rigid}).
We have the following theorem 

\begin{thmsans}[\ref{quasi-strict}]
Every excellent co-Segal $\M$-category is equivalent to an $\M$-category with the same objects and having a strict composition and weak identity morphisms.
\end{thmsans}

As one can observe this is not totally a strictification theorem since it concerns only  the ones we've called \emph{excellent} (Definition \ref{excellent-lax-diag}). We don't know if the theorem holds for all co-Segal categories. 

Even if we don't know examples of non-excellent co-Segal categories, there are some reasons from the theory of \emph{triangulated categories} that suggest that not all co-Segal categories admit a strict model. In fact it was acknowledged to the author that F.Muro has examples of triangulated categories  that don't have \emph{dg-enhancement} (see also \cite{MSS}). But it is more likely that such triangulated category  be can be enhanced by a co-Segal (dg)-category which, a posteriori, shouldn't be equivalent to a strict one.\\

Our theorem goes in the direction of \emph{Simpson's conjecture}  which says that ``higher categories are equivalent to ones that admit a strictly associative composition but weak identities'' (see \cite{Simpson_weak_unit}). A particular case of the conjecture has been proved by Joyal and Kock (see \cite{Joyal_Kock_wu}).

To prove the theorem we simply use the fact a weak equivalence between cofibrant Reedy diagrams induces an equivalence on the colimits. And being \emph{excellent} ensures that we are in this case up-to a weak equivalence.\\

We give a weaker version of the previous result in Theorem \ref{weak-strict}. Unfortunately even this weaker version doesn't not induce a Quillen equivalence between arbitrary co-Segal categories and strict categories. 

Finally we would like to remind the reader that co-Segal categories that are considered here have homotopy units. And that the previous theorem gives a strictification for the composition and not the units. We will address the strictification of homotopy units in a different work. 

\section{Excellent lax diagrams}
\subsection{Preliminaries}
\begin{warn}
In this paper all the set theoretical size issues have been left aside \footnote{We can work with universes $\U \subsetneq \V \subsetneq  \cdots$ }. Some of the material provided here are well known facts and we make no claim of inventing or introducing them. Unless otherwise specified when we say `lax functor' we will mean the ones called \emph{normal lax functors} or \emph{normalized lax functor}. These are lax functors $\Fa$ such that the maps `$\Id \to \Fa(\Id)$' are identities and all the laxity maps $\Fa(\Id) \otimes \Fa(f) \to \Fa( \Id \otimes f)$ are natural isomorphisms.
\end{warn}
In the following $\C$ is a locally Reedy $2$-category (henceforth $\lr$-category) which is simple in the sense of \cite{COSEC1}. This means that each hom-category $\C(A,B)$ has a Reedy structure together with a degree that is compatible with compositions. Consider $\overleftarrow{\C}$ the $2$-category obtained by keeping only  the \ul{inverse} category of each $\C(A,B)$.

\begin{df}
Say that $\C$ is an \textbf{inverse divisible} locally Reedy $2$-category if every composition functor:
$$ 
\overleftarrow{\C}(A,B) \times \overleftarrow{\C}(B,C) \to \overleftarrow{\C}(A,C)
$$
is a Grothendieck fibration.
\end{df}

For a monoidal category $\M=(\ul{M}, \otimes, I)$, we will denote by $\Lax(\C,\M)_n$ the category of normal lax functors and icons; and  by $\kc=\prod_{A,B}\Hom[\C(A,B), \ul{M}]$. \\

We have a forgetful functor: $\Ub: \Lax(\C,\M)_n \to \kc$ that admits a left adjoint if $\M$ is cocomplete (see \cite{COSEC1})\footnote{This hold for arbitrary $2$-categories $\C$ and not only for $\lr$ ones}.\\  

Let  $\F: \C \to \M$ be lax diagram in a complete monoidal category.  Given a $1$-morphism $z\in \C(A,B)$, one has the corresponding notions of:
\begin{enumerate}
\item lax-latching object of $\F$ at $z$: $\laxlatch(\F,z)$;
\item lax-matching  object $\laxmatch(\F,z)=\match(\F_{AB},z)$; 
\item and the classical latching object $\latch(\F_{AB},z)$.
\end{enumerate} 

\begin{rmk}\label{rmk-reedy-map}
We have canonical maps:
\begin{align*}
\laxlatch(\F,z) \to \F(z),\\
\laxmatch(\F,z) \to \F(z),\\
\latch(\F_{AB},z) \to \F(z).
\end{align*}
and one important map:
\begin{align*}
\delta_z:\latch(\F_{AB},z) \to \laxlatch(\F,z).
\end{align*}
We have a factorization of the map $\latch(\F_{AB},z) \to \F(z)$ as:
\begin{align*}
\latch(\F_{AB},z) \xrightarrow{\delta_z} \laxlatch(\F,z) \to  \F(z).
\end{align*} 
\end{rmk}

And if $\M$ is a monoidal model category then we can define corresponding notion of Reedy cofibrations and Reedy fibrations. Denote by $\kc_{\tx{-Reedy}}$  the product Reedy model structure on $\kc=\prod_{A,B}\Hom[\C(A,B), \ul{M}]$. Similarly we will denote by $\kc_{\tx{-proj}}$ the product projective model structure.  \\

 The advantage of having such $\lr$-category is that we can use `Reedy techniques' and establish the following:

\begin{thm}
Let $\M$ be a monoidal model category  and $\C$  be an $\lr$-category which is inverse co-divisible. Then we have:
\begin{enumerate}
\item there exists a unique model structure, called the Reedy model structure, on the category $\Lax(\C,\M)_n$ of normal lax functors such that 
$$ \Ub : \Lax(\C,\M)_n \to \kc_{\tx{-Reedy}} $$
is a right Quillen functor;
\item  if $\C$ is totally direct i.e, $\C=\overrightarrow{\C}$, then we have a `projective' Quillen:
$$ \Ub : \Lax(\C,\M)_n \to \kc_{\tx{-proj}}$$ 
\item if all objects of $\M$ are cofibrant and $\M$ is cofibrantly generated, then for \ul{any} $\C$ we also have a projective Quillen adjunction: 
$$ \Ub : \Lax(\C,\M)_n \to \kc_{\tx{-proj}}.$$ 
between cofibrantly generated model categories.
\end{enumerate}
\end{thm}
 
\begin{proof}
Assertion $(1)$ is the dual statement of 
\cite[Theorem 7.1 ]{Colax_Reedy}. Assertion $(2)$ is a corollary of Assertion $(1)$ combine with the fact that for direct Reedy categories, the projective and Reedy model structures are the same. 
Assertion $(3)$ can be found in a more general context in \cite{COSEC1}.
\end{proof}

\begin{rmk}
To prove Assertion $(3)$ one uses  a transfer lemma of Schwede-Shipley \cite{Sch-Sh-Algebra-module} as exposed in \cite{COSEC1}.
\end{rmk}
\subsection{Excellent lax diagrams}
From now on we will work with the Reedy model structure.
\begin{df}\label{excellent-lax-diag}
A lax diagram $\F \in \Lax(\C,\M)_n$ is $\Ub$-cofibrant if $\Ub(\F)$ is cofibrant in $\prod_{A,B}\Hom[\C(A,B), \ul{M}]$. 
A lax diagram $\F$ is \textbf{excellent} if there is a weak equivalence $\F \xrightarrow{\sim} \G$ where $\G$ is a $\Ub$-cofibrant lax diagram.
\end{df}
Recall that for $\M=( \sset, \times, 1)$, the cofibration are precisely the monomorphisms.
A direct consequence of Remark \ref{rmk-reedy-map} is that:

\begin{prop}
\begin{enumerate}
\item For an arbitrary $\M$, a cofibrant  lax diagram $\F \in \Lax(\C,\M)_n$ in the Reedy structure is excellent if for every $z$ the canonical map
$$ \delta_z:\latch(\F_{AB},z) \to \laxlatch(\F,z)$$
is a cofibration.
\item For $\M=( \sset, \times, 1)$, a cofibrant lax diagram $\F \in \Lax(\C,\M)_n$ in the Reedy structure is excellent \textbf{if and only if} for every $z$, the map $\delta_z$ is a cofibration. 
\end{enumerate}
\end{prop}
\begin{proof}
Indeed being cofibrant in the lax-Reedy structure  the canonical map hereafter is a cofibration:
$$\laxlatch(\F,z) \to \F(z).$$

Therefore  if in addition the maps $\delta_z:\latch(\F_{AB},z) \to \laxlatch(\F,z)$ is also a cofibration, then so is the composite:
$$ \latch(\F_{AB},z) \to \laxlatch(\F,z) \to \F(z).$$

Thus $\F_{AB}$ is Reedy cofibrant for all $(A,B) \in \Ob(\C)^2$ and $\F$ is excellent.\\

Assertion $(2)$ is elementary. Indeed if the composite
$$\latch(\F_{AB},z) \to \laxlatch(\F,z) \to \F(z)$$
is a monomorphism, then so is 
$$\latch(\F_{AB},z) \to \laxlatch(\F,z).$$
\end{proof}

\begin{df}
A morphism $\sigma : \F \to \G$ in $ \Lax(\C,\M)_n $ is an $\Ub$-cofibration if $\Ub(\sigma)$ is a cofibration in $\kc$. 
\end{df}
It follows immediately that $\Ub$-cofibrations are closed under composition and retract. Denote by $\Gamma: \kc \to \Lax(\C,\M)_n $ be the left adjoint of $\Ub$. 
\begin{prop}
Let $\M$ be a model category such that all objects are cofibrant.
With respect to the projective model structure, if for any generating cofibration $\sigma \in \kc$, $\Gamma(\sigma)$ is an $\Ub$-cofibration, then every cofibration in $ \Lax(\C,\M)_n $ is also a $\Ub$-cofibration.
\end{prop}

\begin{proof}[Sketch of proof]
Denote by $\I$ the generating set of cofibrations in $\kc$. By definition of the model structure on $ \Lax(\C,\M)_n$, the set $\Gamma(\I)$ constitutes a set  of generating cofibrations in $ \Lax(\C,\M)_n$ (see \cite{COSEC1}).

Then a cofibration in $ \Lax(\C,\M)_n$ is just a relative $\Gamma(\I)$-cell complex. Therefore it's enough to show that in a pushout square
\[
\xy
(0,18)*+{\Gamma \Aa}="W";
(0,0)*+{\Gamma \Ba}="X";
(30,0)*+{\F \cup^{\Gamma \Aa}\Gamma \Ba}="Y";
(30,18)*+{\F}="E";
{\ar@{->}^-{j}"X";"Y"};
{\ar@{->}^-{\alpha}"W";"X"};
{\ar@{->}^-{}"W";"E"};
{\ar@{->}^-{}"E";"Y"};
\endxy
\]
where $\alpha \in \I$, then the map $\F \to \F \cup^{\Gamma \Aa}\Gamma \Ba $ is also a $\Ub$-cofibration. To calculate that pushout, one starts by taking the pushout between the underlying diagrams in $\kc$.

 And as in any model category, projective cofibrations are closed under pushout; it follows that the first canonical map is also a projective cofibration. This map modifies $\F$ and all the trick is to build the pushout out of that modification. The hypothesis `all objects are cofibrant' is used to guarantee that cofibration are closed by tensor product. 

 In the end the map $\F \to \F \cup^{\Gamma \Aa}\Gamma \Ba $ is a transfinite composition of projective cofibrations and therefore is a projective cofibration. We refer the reader to \cite{COSEC1} for the details on that pushout. 
\end{proof}

\begin{nota}
\begin{enumerate}
\item For each pair $(A,B) \in \Ob(\C)^2$, let $p_{AB}$ be the projection functor:
$$p_{AB}: \kc \to \Hom(\C(A,B), \ul{M}).$$ 
\item $p_{AB}$ has a left adjoint $\delta_{AB}$ which is the `Dirac mass' (see \cite{COSEC1}). 
\item Let $\I_{AB}$ (resp. $\Ja_{AB}$) be a set of generating cofibrations (resp. trivial cofibrations) for  $\Hom(\C(A,B), \ul{M})$.
\end{enumerate}
\end{nota}
By lifting properties and adjunction one can clearly have:
\begin{lem}\label{lem-generation}
\begin{enumerate}
\item The sets 
$$ \coprod_{(A,B)} \{ \delta_{AB}(\alpha); \alpha \in \I_{AB} \}$$
$$ \coprod_{(A,B)} \{ \delta_{AB}(\alpha); \alpha \in \Ja_{AB} \} $$
constitutes a set of generating cofibrations (resp. trivial cofibrations) of $\kc$.
\item Similarly the two sets:
$$ \coprod_{(A,B)} \{ \Gamma[\delta_{AB}(\alpha)]; \alpha \in \I_{AB} \}$$
$$ \coprod_{(A,B)} \{ \Gamma[\delta_{AB}(\alpha)]; \alpha \in \Ja_{AB} \} $$
are generating set of (trivial) cofibrations for $ \Lax(\C,\M)_n$ 
\end{enumerate}
\end{lem}

\subsection{Excellent co-Segal precategories}
From now we take $\C= \sxop$. Recall that $\sxop$ is entirely direct (as $\Depiop$) so the Reedy and projective model structure on both $\kc$ and $\Lax(\C,\M)_n$ are the same.\\

In this case we can explicitly write a formula for the left adjoint $\Gamma$.\\
For $\G \in \prod_{A,B} \Hom[\sx(A,B)^{\tx{op}}, \ul{M}]$, $\Gamma \G$ is given by the formula:
$$\Gamma\G (z)= \G(z) \sqcup (\coprod_{(s_1,..., s_l);   \otimes(s_i)=z ; s_i\neq z} \G(s_1) \otimes \cdots \otimes \G(s_l)).$$

\begin{prop}
Let $\M=(\ul{M}, \otimes, I)$ be a monoidal category having an initial object $0$ and such that for every $m \in M$, $0 \otimes m \cong 0$.\\
 
If $A \neq B$ then for every $\G \in \Hom[\sx(A,B)^{\tx{op}}, \ul{M}]$  we have:
$$\Ub (\Gamma \delta_{AB}\G)\cong \delta_{AB}(\G). $$
\end{prop}
\begin{proof}
Indeed we have $\delta_{AB}(\G)(s)= 0$ if $s\notin \C(A,B)$. Therefore is we have an $l$-tuple $(s_1,..., s_l)$ of composable morphisms such that the composite is $z$ and $s_i\neq z$; then if $A\neq B$, necessarily there is at least one $s_i \notin \C(A,B)$. 

It follows that the only summand in  $\Gamma(\delta_{AB}(\G))(z)$ that is different from $0$ is $\G(z)$ and the proposition follows. 
\end{proof}

\begin{cor}
If $A \neq B$ then for any cofibration $\alpha$ of $\Hom[\sx(A,B)^{\tx{op}}, \ul{M}]$, 
$\Gamma(\delta_{AB}(\alpha))$ is a $\Ub$-cofibration.
\end{cor}
\subsubsection{Obstruction of Excellence} 
From the previous corollary together with Lemma \ref{lem-generation}, it's clear that if  for every $A \in \Ob(\C)$ and any $\alpha \in \I_{AA}$, we have 
$\Gamma(\delta_{AA}(\alpha))$ is a $\Ub$-cofibration; then every cofibration is $\Ub$-cofibration. One can observe that we no longer have $\Ub[\Gamma(\delta_{AA}(G))] \cong \Gamma(\delta_{AA}(G)) $. Indeed if $z=(A,A,A...,A)$, there can be non trivial summand that contain a tensor product in $\Gamma(\delta_{AA}(G))(z)$. 

\begin{rmk}
It's precisely the presence of \emph{algebraic data} that \emph{kills} the `projectiveness' of cofibrations.  The main reason is that the category $\Ar(\M)$ of arrows of $\M$ with it's projective model structure; cofibration are not (necessarily) closed under tensor product.  
\end{rmk}

One can establish the following.
\begin{prop}
If $\M=(\ul{M}, \otimes, I)$ is a monoidal model category such that all objects are cofibrant then:

equivalent. 
\begin{enumerate}
\item In the adjunction $$\Ub: \Lax[\sxop,\M]_n  \rightleftarrows \prod_{A,B}\Hom[\sx(A,B)^{\tx{op}}, \ul{M}]: \Gamma$$ 
for every cofibrant object in $\G \in \kc$ then $\Gamma(\G)$ is $\Ub$-cofibrant and hence excellent. It follows that for any $\G$, $\Gamma \G$ is excellent. 
\item Every $\M$-category is excellent
\end{enumerate}
\end{prop}

\begin{proof}
Assertion $(2)$ is clear since all object are cofibrant. In fact  $\M$-categories correspond to the locally constant lax diagrams. And given a category with an initial object $E$, e.g $\sx(A,B)^{op}$ with $E=(A,B)$, constant diagrams correspond to the essential image of the left adjoint $\Fb^{E}$ of the evaluation at $E$. And since $\Fb^{E}$ is a left Quillen functor with respect to the projective model structure, then the result follows.\\
One can alternatively check this by lifting properties agains all fibrations.\\

For Assertion $(1)$ we proceed as follows.
Let $z$ be a $1$-morphism in $\sxop(A,B)$. Recall that $z$ is a sequence $(A,..., A_i, ...,B)$.  If $\partial_z$ represents the classical \emph{latching category} at $z$; then the particularity of $\sxop$ is that:
\begin{claim}
For any $1$-morphism $z$ of $\sxop$, and any presentation $(s_1,..., s_l)$ of $z$,  we have an isomorphism:
$$\partial_z \cong  \partial_{s_1} \times \cdots \times \partial_{s_l} $$
\end{claim}
In fact we leave the reader to verify that this is true in any \emph{direct divisible} $\lr$-category (see \cite{Colax_Reedy}). And thanks to this isomorphism and from the formula of $\Gamma$ one has that for any $\G$:
$$\latch(\Gamma\G,z) \cong \latch(\G,z) \sqcup \coprod_{(s_1,..., s_l);   \otimes(s_i)=z ; s_i\neq z} \latch(\G,s_1) \otimes \cdots \otimes \latch(\G,s_l)$$

If $\G$ is cofibrant then, by definition every map
$$\latch(\G,s_i) \to \G(s_i) \hspace{0.5in} \tx{(including $s_i=z$)}$$
is a cofibration (with cofibrant domain). And since cofibrations with cofibrant domain are closed under tensor product and coproduct, one clearly have that the canonical map:
$$\latch(\Gamma\G,z) \to (\Gamma\G)(z)  $$ 
is also a cofibration.
\end{proof}
\section*{New model structure for precategories}
In the following we still work with $\C=\sxop$ for some set $X$. Let's write for simplicity $\msx=\Lax[\sxop,\M]_n$. We assume that $\M$ is a combinatorial monoidal model category. It can be shown that the projective (=Reedy) model structure in the previous sections is also combinatorial (see \cite{COSEC1}).\\

In this section we are going to construct another model structure on $\msx$ that will be used in the upcoming sections. We will use Smith's theorem (see for example \cite{Hov-model}).\\

Consider the following  maps in $\msx$.
\begin{enumerate}
\item Let $\I_{ex}$ be the set of maps $ \coprod_{A\neq B} \{ \Gamma[\delta_{AB}(\alpha)]; \alpha \in \I_{AB} \}$;
\item  Let $\Ja_{ex}$ be the set of maps $ \coprod_{A\neq B} \{ \Gamma[\delta_{AB}(\alpha)]; \alpha \in \Ja_{AB} \} $;
\item Let $\Wa$ be the class of maps $\sigma$ such that  for $A\neq B$, the component $\Ub(\sigma)_{AB}$ is a level-wise weak equivalence.  We will identify $\Wa$ with the subcategory generated by these maps in $\msx$. 
\end{enumerate}
It follows that any (old) weak equivalence in $\msxproj$ is in $\Wa$ and that $\Wa$ is closed by composition and retract. 

Using the fact we already have a model structure on $\msx$  and thanks to Smith's theorem one has:

\begin{thm}\label{thm-new-model}
There is a cofibrantly generated model structure on $\msx$ with :
\begin{enumerate}
\item $\I_{ex}$ as the set of generating cofibrations;
\item $ \Ja_{ex}$ as the set of generating trivial cofibrations;
\item $\Wa$ as the subcategory of weak equivalences.
\end{enumerate}
The model structure is combinatorial and will be denoted by $\msxex$.\\ 

The identity functor $\msxproj \to \msxex$ is a right Quillen functor.
\end{thm}

\begin{proof}
All the criterions of Smith's theorem are easily verified. Indeed since maps in $\Ja_{ex}$ are old trivial cofibrations, the pushout along a map in $\Ja_{ex}$ is an old trivial cofibration and in particular an old weak equivalence, thus in $\Wa$.\\

Finally (trivial) fibrations in $\msxproj$ are also new (trivial) fibrations since we have smaller set of generating (trivial) cofibrations.
\end{proof}

\section{Locally constant lax functors and enriched categories}
\subsubsection{Indiscrete or coarse category} Recall that the `object functor' 
$\Ob: \Cat \to \Set$ that takes a category $\Ba$ to it set of objects $\Ob(\Ba)$, has a left adjoint $\disc : \Set \to \Cat$ `the discrete functor'. It turns out that this functor has also a right adjoint $ \indisc : \Set \to \Cat$.  We will denote by $\ol{X}:= \indisc(X)$.  By definition for any category $\Ba$ and any set $X$ we have an isomorphism of sets:
$$\Hom(\Ba,\ol{X}) \cong \Hom(\Ob(\Ba),X)$$
functorial in $X$ and $\Ba$;
where the left-hand side is the set of functors from $\Ba$ to $\ol{X}$ while the right-hand side is the set of functions from $\Ob(\Ba)$ to $X$.  Below we give a brief description of $\ol{X}$.
\paragraph{Brief description of $\ol{X}$} The category $\ol{X}$ is the terminal \ul{connected} groupoid having $X$ as the set of objects. There is precisely a unique morphism between any pair of elements: $$\ol{X}(a,b)=\Hom_{\ol{X}}(a,b):= \{(a,b)\} \cong 1.$$
The composition is the unique one: the bijection $1 \times 1 \cong 1$. Given a function $g: \Ob(\Ba) \to X$, the associated functor  $\ol{g}: \Ba \to \ol{X}$ is given by the (unique)  constant functions  
$$\ol{g}_{UV}: \Ba(U,V) \to \ol{X}(g(U),g(V))\cong 1.$$
 
\begin{rmk}
 If $X$ has two elements then $\ol{X}$ is the ``\textbf{walking-isomorphism category}'' in the sense that any isomorphism in a category $\Ba$ is the same thing as a functor $\ol{X} \to \Ba$. 
\end{rmk}

\subsection{An adjunction lemma}
In classical $1$-category theory, given an indexing category $\J$ one define the colimit functor 
$$\colim: \Hom(\J,\M) \to \M$$ as the left adjoint of the constant functor $\cb_{\ast}: \M \to \Hom(\J,\M)$.\\
Below we extend, locally, this fact to (normal) lax-functor when $\J$ is a $2$-category. \\

\begin{df}
Say that a (normal) lax functor $ \Fa: \J \to \M$ is \textbf{locally constant} if for every $(i,j)$ the component $\Fa_{ij}: \J(i,j) \to \M$ is a constant functor.
\end{df}
The reader can check that such a (normal) lax functor is the same thing as a (semi) $\M$-category whose set of objects is $\Ob(\J)$; the hom-object between $i$ and $j$ is the value of $\Fa_{ij}$.\ \\

Denote by $\cb_{\ast}\Lax(\J,\M) \hookrightarrow \Lax(\J,\M)$ the full subcategory of locally constant lax functors and transformations which are icons (see \cite{Lack_icons}).

\begin{lem}\label{lax-to-cat}
The inclusion functor $\cb_{\ast}\Lax(\J,\M) \hookrightarrow \Lax(\J,\M)$ has a left adjoint.
\end{lem}
We can rephrase the above adjunction using the previous observation that we have an adjunction $\Ob \dashv \indisc$ that is also valid for $2$-categories. This means that for any $2$-category $\J$ and any (nonempty) set $X$ then we have an isomorphism of \ul{sets}:
$$2\tx{-Func}(\J, \ol{X}) \cong \Hom[\Ob(\J),X].$$

The unit of this adjunction (when $X= \Ob(\J)$) gives a canonical $2$-functor $$\varepsilon_{\J}: \J \to \ol{\Ob(\J)}.$$
Then the lemma says essentially that the pullback functor 
$$\varepsilon_{\J \ast}: \Lax[\ol{\Ob(\J)}, \M] \to \Lax(\J,\M)$$
has a left adjoint 
$$\varepsilon_{\J !}: \Lax(\J,\M) \to \Lax[\ol{\Ob(\J)}, \M]$$  
when $\M$ is cocomplete.

\begin{note}
This situation is a left Kan extension for lax functors and there is a general statement for $2$-functors $\J \to \J'$ but we will not go through that here (see \cite{COSEC1}).
\end{note}

\begin{proof}[Sketch of proof]
For a lax functor $\Fa$ one construct the adjoint-transpose by taking the colimit of each component $\Fa_{ij}$ of $\Fa$. Let $m(i,j):= \colim \Fa_{ij}$. As $\M$ is monoidal \ul{closed}, colimits distribute over $\otimes$.

 Consider the following compatible diagram which ends at $m(i,k)$:
$$\Fa_{ij}s \otimes \Fa_{jk} t  \to \Fa_{ik}(s \otimes t) \to m(i,k)$$
in which $(s,t)$ runs through $\J(i,j) \times \J(j,k)$. \\

We get a unique map by universal property of the colimit:
$$\varphi: m(i,j) \otimes m(j,k) \to m(i,k).$$

For the coherence axiom, one considers the compatible diagram  ending at $m(i,l)$ as $(s,t,u)$ runs through $\J(i,j) \times \J(j,k) \times \J(k,l)$:
$$\Fa_{ij}s \otimes \Fa_{jk} t  \otimes \Fa_{kl}u \to \Fa_{il}(s \otimes t \otimes u) \to  m(i,l).$$

Note that there are two ways to go from $\Fa_{ij}s \otimes \Fa_{jk} t  \otimes \Fa_{kl}u$ to $ \Fa_{il}(s \otimes t \otimes u)$ and the coherence for $\Fa$ says that the two ways induce the same map. 

By the universal property of the colimit we get a unique map that makes every thing compatible: 
$$\gamma_1: m(i,j) \otimes m(j,k) \otimes m(k,l) \to m(i,l).$$
On the other hand we have two other maps in $\Hom[m(i,j) \otimes m(j,k) \otimes m(k,l), m(i,l)]$:
\begin{enumerate}
\item $\gamma_2: m(i,j) \otimes m(j,k) \otimes m(k,l) \xrightarrow{\varphi \otimes \Id} m(i,k) \otimes m(k,l) \xrightarrow{\varphi} m(i,l)$
\item $\gamma_3: m(i,j) \otimes m(j,k) \otimes m(k,l) \xrightarrow{ \Id \otimes \varphi} m(i,j) \otimes m(j,l) \xrightarrow{\varphi} m(i,l)$
\end{enumerate}

If we restrict these two maps to $\Fa_{ij}s \otimes \Fa_{jk} t  \otimes \Fa_{kl}u$, one gets the same compatible diagram; thus by uniqueness of the map out of the colimit we get that $\gamma_1= \gamma_2= \gamma_3$ and the coherence axiom follows. We leave the reader to check that the unit axiom holds also.
\end{proof}

We will make an abuse of notation and write $\mcatx$ be the category of semi $\M$-categories with fixed set of objects $X$. We endow $\mcatx$ with its canonical model structure (= projective). A direct consequence of the previous lemma is that:
\begin{cor}
\begin{enumerate}
\item We have a Quillen adjunction 
$$| |: \msxproj \leftrightarrows \mcatx: \iota$$
where $| |$ is left Quillen and $\iota$ is right Quillen.
\item We also have a Quillen adjunction:
$$| |: \msxex \leftrightarrows \mcatx: \iota$$ 
with $\iota$ still right Quillen.  
\end{enumerate}
\end{cor}

\begin{proof}
Indeed  $\mcatx$ is equivalent to $\cb_{\ast}\Lax[\sxop,\M]$. In both $\mcatx$ and $\msxproj$ (trivial) fibrations are level-wise and we get Assertion $(1)$. Assertion $(2)$ is obvious and follows from Theorem \ref{thm-new-model}.
\end{proof}

\begin{rmk}
From the lemma we know that for any co-Segal pre-category $\F$ there is a semi-$\M$-category $|\F|$ which is the adjoint transpose of $\F$. The co-unit of the adjunction if a transformation $\sigma: \F \to |\F|$ of lax-morphisms.
\end{rmk}
A natural question is to ask whether or not the canonical map $\sigma: \F \to |\F|$ a weak equivalence. If this map is a weak equivalence then we will say that that we have a (semi) strictification of $\F$. We treat this question in the next section.

%%%%%%%%%%%%%%%%%%%%%%% 

\section{Quasi-strictification}
 Given a co-Segal $\M$-category $\F$ a natural candidate to consider is $|\F|$ constructed previously. \ \\
The map $\sigma: \F \to |\F|$ will be a weak equivalence if and only if we can show that for every $(A,B)$ , the canonical map $\F(A,...,B) \to \colim \F_{AB}$ is a weak equivalence. This problem can be formulate in general as follow
\begin{quest}
Given a diagram $ \F: \J \to \M $ such that $\F(i) \to \F(j)$ is a weak equivalence for all morphism $i \to j$ of $\J$; is the canonical map $\F(i) \to \colim \F$ a weak equivalence ?
\end{quest}

The answer to that question is negative in general as illustrated in the following example.
\begin{ex} 
The coequalizer hereafter is not equivalent to all other objects:
\[
\xy
(0,0)*++{1}="X";
(20,0)*++{[0,1]}="Y";
(40,0)*++{S^1}="Z";
{\ar@<-0.5ex>@{->}_-{0}"X";"Y"};
{\ar@<0.5ex>@{->}^-{1}"X";"Y"};
{\ar@{->}^{p}"Y";"Z"};
\endxy
\]
\end{ex}
As we shall see in a moment there are some cases where we have an affirmative answer. More precisely we have:

\begin{prop}\label{prop_weak_equiv}
Let $\M$ be a model category and $\J$ be a Reedy  category,  with an initial object $e$.\\ 

Let  $\F: \J \to \M$ be a Reedy cofibrant diagram such that for every morphism $i \to j$ of $\J$, the map $\F(i) \to \F(j)$ is a weak equivalence. Then every canonical map $\F(i) \to \colim\F$ is a weak equivalence. 
\end{prop}

\begin{proof}
By $3$-for-$2$ it's enough to have that $\F(e) \to \colim \F$ is a weak equivalence.\\

As $\J$ has an initial object, then automatically $\J$ has cofibrant constant in the sense of \cite[Def. 15.10.1]{Hirsch-model-loc}. Now If $\F$ is Reedy cofibrant then necessarily $\F(e)$ is cofibrant in $\M$ since the latching category of $\J$ at $e$ is empty. It follows that the constant diagram $\cb_{\ast}(\F(e))$ is Reedy cofibrant.\\

Now as $e$ is initial, we have a canonical natural transformation $\eta: \cb_{\ast}(\F(e)) \to \F $ which is a point-wise weak equivalence of Reedy cofibrant diagrams. Consequently taking the colimit preserve weak equivalences, thus $\F(e) \to \colim \F$ is a weak equivalence in $\M$.  
\end{proof}

\subsubsection{Quasi-strictification}
\begin{df}

Say that a co-Segal category $\C: \sxop \to \M$ \textbf{has weak identities} if the semi-$\Ho(\M)$-category  
$$\C: \sxop \to \M \to \Ho(\M)$$
has identities.
\end{df}

\begin{thm}\label{quasi-strict}
Every \textbf{excellent} co-Segal $\M$-category with weak identities is \textbf{weakly co-Segal equivalent} to an $\M$-category with weak identities.
\end{thm}
The proof of the theorem is a direct application of Proposition \ref{prop_weak_equiv}. 

\begin{proof}[Proof of Theorem \ref{quasi-strict}]
Let $\F: \sxop \to \M$ be a unital co-Segal category. Since being unital is stable under weak equivalence we can assume that $\F$ is $\Ub$-cofibrant.
As $\Ub(\F)$ is projective cofibrant in $\prod_{(A,B) \in X^2} \Hom(\sx(A,B)^{op}, \M)$, this means that each component $\F_{AB} \in \Hom(\sx(A,B)^{op}, \M)_{proj}$ is projective cofibrant. Being projective cofibrant allows computing the homotopy colimit of $\F_{AB}$ as the usual colimit of $\F_{AB}$.\ \\

By Lemma \ref{lax-to-cat} we get a semi-$\M$-category $|\F|$ by declaring $|\F|(A,B):= \colim \F_{AB}$. $|\F|$ is a locally constant object of $\msx$ equipped with a canonical map $\sigma: \F \to |\F|$ in $\msx$.  Thanks to Proposition \ref{prop_weak_equiv}, all canonical maps $\Fa(A...,B) \to \colim \F_{AB}$ are weak equivalences; these maps are exactly the components of $\sigma: \F \to |\F|$ which means that $\sigma$ is a weak equivalence in $\msx$. \ \\
\ \\
$|\F|$ is a strict semi-category with a strict composition; it inherits of the (weak) unities of $\F$ since $\sigma$ is a weak equivalence and the theorem follows.
\end{proof}
We have an immediate consequence.
\begin{cor}
Let $\msx_{\Ub\tx{-cof}}\hookrightarrow \msx$ be the full subcategory of $\Ub$-cofibrant co-Segal categories.  Then if all objects of $\M$ are cofibrant, the restrict adjunction 
$$| |: \msx_{\Ub\tx{-cof}} \leftrightarrows \mcatx: \iota$$
induces an equivalence between the respective homotopy categories. 
\end{cor}
\begin{proof}[Proof of the corollary]
Let $\F$ be an $\Ub$-cofibrant co-Segal category and $\Aa$ be a category. Given a morphism $\sigma: \F \to \Aa$ in $\msx$, then by adjunction we can factorize that map as:
$$ \F \to |\F| \xrightarrow{\sigma'} \Aa .$$

From the proof of the theorem we know that the canonical map $ \F \to |\F| $ is always a weak equivalence if $\F$ is an $\Ub$-cofibrant co-Segal category. Then by $3$-for-$2$, we get that $\sigma$ is a weak equivalence in $\msx$ if and only if $\sigma': |\F| \to \Aa$ is a weak equivalence in $\mcatx$.

Moreover both functors $\iota$ and $||$ preserve weak equivalences and $|| \circ \iota= \Id$.  
The rest is just a categorical argument on localization of categories. 
\end{proof}

The next move is to go from quasi-strictification to strictification. This is a general issue and will be discussed in full generality in a different work.

The previous theorem has a weaker version using the model structure $\msxex$. 

\begin{thm}\label{weak-strict}
 Every co-Segal category in $\msxex$ is equivalent to a strict one.   
\end{thm}

\begin{proof}
Let  $\F$ be a co-Segal category. Then up to a cofibrant replacement we can assume that $\F$ is cofibrant. Note that such cofibrant replacement is only, a priori, partially co-Segal. 
By definition of the model structure on $\msxex$, since $\F$ is cofibrant then for $A \neq B$ the functor $$\F_{AB}: \sx(A,B)^{op} \to  \ul{M}$$ is projective cofibrant and take its values in the subcategory of weak equivalences (partial co-Segal conditions). 

Then from Proposition \ref{prop_weak_equiv} we get that for $A\neq B$
all canonical map $$\F(A,...B) \to |\F|(A,B)$$ are weak equivalences. This means that $\F \to |\F|$ is a weak equivalence in $\msxex$.
\end{proof}

\section{Commutative co-Segal monoids}
\subsection{Preliminaries} Following Leinster \cite{Lei2}, we will denote by $\Phi$ the skeletal category of finite sets: its objects are finite sets $n=\{ 0,...,n-1\}$ for each integer $n \geq 0$; and its morphisms are all functions. $\Phi$ has a monoidal structure given by disjoint union, which is a symmetric operation. So we have a symmetric monoidal category $(\Phi,+,0)$, where $+$ is the disjoint union and $0$ is the empty set. \ \\

Let $\Gamma$ be the category considered by Segal in  \cite{Seg1}. The objects of $\Gamma$ are all finite sets, and a morphism from $S$ to $T$ is a morphism from $S$ to $P(T)$, the set of subsets of $T$.\ \\

Leinster \cite[Prop 3.1.1]{Lei2} pointed out a relationship between $\Phi$ and $\Gamma$ in the following proposition.
\begin{prop}
Let $\M= (\ul{M}, \times, 1)$ be a category with finite product.  Then there is an isomorphism of categories:

$$\scolax[(\Phi,+,0), (\ul{M}, \times, 1)]  \cong[ \Gamma^{op}, \ul{M}].$$  
\end{prop}
Here `$\scolax$' stands for symmetric colax monoidal functors. Following the above result and the Segal formalism, Leinster considered weak commutative algebra (or monoid) in a symmetric monoidal  category $\M=(\ul{M},\otimes,I)$ having a subcategory of weak equivalence $\W$ which satisfies certain properties; we called in \cite{SEC1}, the pair $(\M,\W)$, a \emph{base of enrichment}. The following definition is due to Leinster. 

\begin{df}
Let $\M=(\ul{M},\otimes,I)$ be a symmetric monoidal category with a subcategory $\W$ such that the pair $(\M,\W)$ is a base of enrichment.\ \\

A \textbf{homotopy commutative monoid} in $\M$ is a symmetric colax monoidal functor:
$$ \C: (\Phi,+,0) \to \M $$
satisfying the Segal conditions:
\begin{enumerate}
\item for every $m,n \in \Phi$ the colaxity map $\C(n+m) \to \C(n) \otimes \C(m)$ is a weak equivalence;
\item the map $\C(0) \to I$ is a weak equivalence.  
\end{enumerate}  
\end{df}

\subsection{The co-Segal formalism} 
Colax diagrams are difficult to manipulate for a homotopical and categorical point of view. For example computing limits in the category $\scolax[(\Phi,+,0), (\ul{M}, \otimes, 1)]$ is not straightforward !\ \\

For this reason we will change colax to lax using the co-Segal formalism. Let $\phepi$ be the subcategory of $\Phi$ having the same objects but only morphisms which are surjective. $\phepi$ is the `symmetric' companion of $\Depi$. It's easy to see that that we a symmetric monoidal subcategory $(\phepi,+,0) \subset (\Phi,+,0).$ We have an obvious (nonsymmetric) monoidal functor 
$$i: (\Depiop,+,0) \to (\phepiop,+,0).$$

\begin{df}
Let $\M=(\ul{M},\otimes,I)$ be a symmetric monoidal category with a subcategory $\W$ such that the pair $(\M,\W)$ is a base of enrichment.\ \\

A  \textbf{commutative co-Segal semi-monoid} in $\M$ is a normal \textbf{symmetric lax monoidal} functor:
$$ \C: (\phepiop,+,0) \to \M $$
such that  for every map $f:n \to m$ of  $\phepi$,  the structure map $$ \C(m) \to \C(n)$$
is a weak equivalence. \ \\

A \textbf{ commutative co-Segal monoid} is a commutative co-Segal semi-monoid $\C$ such that the induced diagram
$$i^{\star}\C: (\Depiop,+,0) \to \M$$
is a co-Segal monoid.
\end{df}

If $\C$ is a commutative co-Segal monoid, then as in the noncommutative case the monoid structure is on the object $\C(1)$. The commutative quasi-multiplication is obtained as before.\\

A direct consequence of the result of the previous section is:
\begin{thm}\label{quasi-strict-com}
Every  commutative excellent co-Segal monoid is \textbf{weakly co-Segal equivalent} to a commutative one which is strictly associative.
\end{thm}

\appendix

\bibliographystyle{plain}
\bibliography{Bibliography_These}

\begin{thebibliography}{10}

\bibitem{SEC1}
H.~V. {Bacard}.
\newblock {Segal Enriched Categories I}.
\newblock http://arxiv.org/abs/1009.3673.

\bibitem{COSEC1}
H.~V. {Bacard}.
\newblock {Lax Diagrams and Enrichment}.
\newblock June 2012.
\newblock http://arxiv.org/abs/1206.3704.

\bibitem{Colax_Reedy}
Hugo~V. Bacard.
\newblock Colax reedy diagrams.
\newblock 2013.

\bibitem{Bergner_rigid}
Julia~E. Bergner.
\newblock Rigidification of algebras over multi-sorted theories.
\newblock {\em Algebr. Geom. Topol.}, 6:1925--1955, 2006.

\bibitem{Hirsch-model-loc}
Philip~S. Hirschhorn.
\newblock {\em Model categories and their localizations}, volume~99 of {\em
  Mathematical Surveys and Monographs}.
\newblock American Mathematical Society, Providence, RI, 2003.

\bibitem{Hov-model}
Mark Hovey.
\newblock {\em Model categories}, volume~63 of {\em Mathematical Surveys and
  Monographs}.
\newblock American Mathematical Society, Providence, RI, 1999.

\bibitem{Joyal_Kock_wu}
Andr{\'e} Joyal and Joachim Kock.
\newblock Weak units and homotopy 3-types.
\newblock In {\em Categories in algebra, geometry and mathematical physics},
  volume 431 of {\em Contemp. Math.}, pages 257--276. Amer. Math. Soc.,
  Providence, RI, 2007.

\bibitem{Lack_icons}
Stephen Lack.
\newblock Icons.
\newblock {\em Appl. Categ. Structures}, 18(3):289--307, 2010.

\bibitem{Lei2}
T.~{Leinster}.
\newblock {Homotopy Algebras for Operads}.
\newblock http://arxiv.org/abs/math/0002180.

\bibitem{MSS}
Fernando Muro, Stefan Schwede, and Neil Strickland.
\newblock Triangulated categories without models.
\newblock {\em Invent. Math.}, 170(2):231--241, 2007.

\bibitem{Sch-Sh-Algebra-module}
Stefan Schwede and Brooke~E. Shipley.
\newblock Algebras and modules in monoidal model categories.
\newblock {\em Proc. London Math. Soc. (3)}, 80(2):491--511, 2000.

\bibitem{Seg1}
Graeme Segal.
\newblock Categories and cohomology theories.
\newblock {\em Topology}, 13:293--312, 1974.

\bibitem{Simpson_weak_unit}
C.~{Simpson}.
\newblock {Homotopy types of strict 3-groupoids}.
\newblock {\em ArXiv Mathematics e-prints}, October 1998.

\end{thebibliography}
\end{document}